\documentclass[a4paper,10pt]{article}
\usepackage{amssymb}
\usepackage{amsmath,amsfonts,amsthm,amssymb}
\usepackage[dvips]{graphics}
\usepackage{epsfig}
\usepackage{indentfirst}
\usepackage{cancel,soul}
\usepackage{ulem}

\usepackage{color}
\usepackage{bm}
\allowdisplaybreaks

\usepackage[colorlinks=true,citecolor=red,linkcolor=blue,%
urlcolor=RubineRed,pdfpagetransition=Blinds,pdftoolbar=false,pdfmenubar=false]{hyperref}

\newtheoremstyle{theorem}
  {10pt}          
  {10pt}  
  {\sl}  
 {}
  {\bf}  
  {. }    
  { }    
  {}     
\theoremstyle{theorem}

\newtheorem{theorem}{Theorem}[section]

\newtheorem{definition}{Definition}[section]
 \newtheorem{lemma}{Lemma}[section]
 
 \newtheorem{remark}{Remark}[section]
  
\numberwithin{equation}{section}

\newtheoremstyle{defi}
  {10pt}          
  {10pt}  
  {\rm}  
  {}  
  {\bf}  
  {. }    
  { }    
  {}     
\theoremstyle{defi}

\begin{document}
\baselineskip = 16pt

\title{\bf Inviscid Incompressible Limit for Degenerate Compressible Navier-Stokes Equations on Expanding Domains}

\author{
Tong Tang\footnote{School of Mathematical Science, Yangzhou University, Yangzhou, P.~R.~China. E-mail: tt0507010156@126.com}\\
Emil Wiedemann \footnote{Department of Mathematics, Friedrich-Alexander-Universit\"at Erlangen-N\"{u}rnberg, Erlangen, Germany. \mbox{E-mail}: emil.wiedemann@fau.de}\\
Lu Zhu \footnote{School of Mathematics, Hohai University, Nanjing, P.~R.~China. E-mail: zhulu@hhu.edu.cn}\\
\date{}}

\maketitle

\begin{abstract}
We simultaneously investigate the inviscid and Low Mach number limits on expanding domains for the $3D$ degenerate compressible Navier-Stokes equations, whose viscosity depends on density. Starting from ill-prepared data, we show the limit system is the incompressible Euler system. 
Our result extends a previous result of Feireisl et al.\ concerning the constant viscosity Navier-Stokes equations, and is the first to establish the inviscid and incompressible limits at the same time for density-dependent viscosity.  

\vspace{0.5cm}


\medskip

\end{abstract}

\maketitle
\section{Introduction}\setcounter{equation}{0}
Mathematical models of fluid flow often contain parameters that describe to what extent the fluid is, say, compressible, viscous, Newtonian, homogeneous, etc. One wishes that there be, on a metalevel, a certain kind of continuity with respect to such parameters: This would allow, for instance, to effectively describe a fluid with very small viscosity using a perfectly inviscid model, or a very slightly compressible fluid by means of an incompressible model. This is, generally speaking, the problem of \emph{singular limits}, where the attribute `singular' refers to situations where crucial quantities or estimates blow up as the parameters converge. This is most notably the case for the inviscid limit, where any information on the velocity gradient is lost as viscosity tends to zero.

Here, we tackle three singular limits at once: We let viscosity and Mach number tend to zero and, at the same time, the space domain will exhaust the whole space $\mathbb R^3$.
More precisely, we consider the degenerate compressible Navier-Stokes equations with drag on a spatial domain $\Omega_M$:

\begin{equation}\label{1.1}
\left\{
\begin{array}{llll}
\partial_t\rho+\text{div}_{x}(\rho\mathbf{u})=0,\\
\partial_{t}(\rho\mathbf u)+\text{div}_{x}(\rho\mathbf{u}\otimes\mathbf{u})+\frac{1}{\varepsilon^2}\nabla_{x}p(\rho)
=\varepsilon^\alpha\text{div}_{x}(\rho D_x(\mathbf{u}))-r_1(\varepsilon)\rho|\mathbf u|\mathbf u,
\end{array}\right.
\end{equation}
where $\mathbf u(t,x)\in\mathbb{R}^3$ and $\rho(x,t)$ 
represent the velocity and density, respectively. Moreover, the pressure is given constitutively as a function of density by $p(\rho)=\rho^\gamma$ for some $\gamma>1$, $\alpha>0$ is given,  and $r_1(\varepsilon)$ is the coefficient of friction of which we assume $r_1(\varepsilon)>0$ for each $\varepsilon>0$ and $r_1(\varepsilon)\rightarrow0$ as $\varepsilon\rightarrow0$. 

The family of domains $\{\Omega_M\}$ has the following properties:
\begin{enumerate}
\item[(H1):] $\Omega_M\subset\mathbb{R}^3$ {are simply connected, bounded} $C^2$ domains,
\item[(H2):] {There exists} $\omega>0$ {such that} $\{x\in\mathbb{R}^3: |x|<\omega M\}\subset\Omega_M$,
\item[(H3):] There exists $\beta>0$ such that $|\partial\Omega_M|_2\leq\beta M^2$,
\item[(H4):] $\varepsilon M(\varepsilon)\rightarrow\infty$ {as} $\varepsilon\rightarrow0$,
\end{enumerate}
where $|\cdot|_2$ denotes the two-dimensional Hausdorff measure. The boundary condition is
\begin{align}\label{1.2}
\rho\mathbf u|_{\partial\Omega_M}=0,\hspace{5pt}\nabla_x\rho\times n|_{\partial\Omega_M}=0
\end{align}
and we consider so-called ill-prepared initial data
\begin{align}\label{1.4}
&\rho|_{t=0}=\rho_{0,\varepsilon}=1+\varepsilon\rho^{(1)}_{0,\varepsilon},\hspace{5pt}\mathbf u|_{t=0}=\mathbf u_{0,\varepsilon},\\
&\rho^{(1)}_{0,\varepsilon}\rightarrow\rho^{(1)}_{0}\hspace{3pt}\text{in}\hspace{3pt}L^2(\mathbb{R}^3),
\hspace{5pt}\mathbf u_{0,\varepsilon}\rightarrow\mathbf u_0\hspace{3pt}\text{in}\hspace{3pt}L^2(\mathbb{R}^3),
\end{align}
where $\mathbf u_0$ is not necessarily divergence-free.

Regarding the important question of well-posedness of compressible Navier-Stokes equations, smooth or classical solutions for the compressible Navier-Stokes equations are known to exist locally in time. For initial data representing small perturbations of equilibrium in a suitable sense, one can obtain global existence of strong solutions. 

Turning to weak solutions, P.-L.~Lions \cite{lion} and Feireisl \cite{e1} showed global existence for the constant viscosity compressible Navier-Stokes equations for adiabatic exponents greater than $3/2$. In contrast to the constant viscosity case, the equations become more complicated when the viscosity depends on density, because there is no sufficient information on $\nabla_x\mathbf u$ due to the degeneracy of the momentum equation, a problem that occurs near or at vacuum. Vasseur \& Yu \cite{v2} and Li \& Xin \cite{li2} independently proved the global existence of weak solutions for the compressible degenerate Navier-Stokes equations in $\mathbb{T}^3$ by virtue of the Bresch-Desjardins entropy introduced in~\cite{BDentropy}, among other techniques.
Previously, Bresch et al.~\cite{b1} had already established existence results on a bounded domain under additional restrictions on the initial density.

The literature on low Mach number and vanishing viscosity for compressible flows is vast.
However, most results concern constant viscosity. To our knowledge, there are only a few results on the compressible degenerate Navier-Stokes system. Fanelli and Zatorska \cite{f} and, more recently, Chaudhuri et al.~\cite{n} studied the anelastic incompressible limit for compressible degenerate Navier-Stokes equations with fixed positive viscosity coefficients. Bisconti and Caggio \cite{bi}, on the other hand, proved the vanishing viscosity limit for compressible degenerate Navier-Stokes equations at fixed Mach number. Both \cite{bi,n} utilized the revised relative entropy inequality for degenerate Navier-Stokes equations based on~\cite{b2}. 

The relative entropy method, generally speaking, was introduced by Dafermos \cite{D} for hyperbolic systems. Germain~\cite{Germain} and then Feireisl et al.~\cite{e3} adopted the idea and proved weak-strong uniqueness for the constant viscosity compressible Navier-Stokes equations. The relative entropy/energy method is a versatile and robust tool which is often applied to establish singular limits of weak to strong solutions. For details and applications, see for instance~\cite{e2,e4,e5,w}.  Bisconti et al.~\cite{bi2} used the same relative energy functional as in the constant viscosity case \cite{e3} to prove the vanishing viscosity limit for compressible degenerate Navier-Stokes equations.

In this paper, we consider the limit for~\eqref{1.1} as $\varepsilon\rightarrow0$, $M\rightarrow\infty$, corresponding to the triple limit of vanishing viscosity, low Mach number, and expanding domains. The latter limit has been studied, for other equations, in~\cite{k, e6} as a way to rigorously show the independence of local interior behaviour of the flow from distant boundaries, which otherwise would considerably influence the dynamics. We thus obtain as our main result (Theorem~\ref{thm2.1}) the convergences
$\rho\rightarrow1$, $\mathbf u\rightarrow\mathbf v$, where $\mathbf v$ is a strong solution of the incompressible Euler system
\begin{align}\label{1.8}
\partial_t\mathbf v+\mathbf v\cdot\nabla_x\mathbf v+\nabla_x\Pi=0,\nonumber\\
\text{div}_x\mathbf v=0, \hspace{3pt}\text{in}\hspace{3pt}\mathbb{R}^3,
\end{align}
as long as the latter exists. As is typical for the relative energy method, our result is thus of type `weak-to-strong convergence'.

It is well known that the obstacle to the asymptotic limit $\varepsilon\rightarrow0$, as far as the Mach number is concerned, lies in the oscillations due to the acoustic velocity component, and, as far as the viscosity is concerned, in the possible formation of boundary layers. To avoid problems related to the latter, we consider the viscous equations in a large domain and the target system in the whole space. This is the above-mentioned expanding domain scenario. It comes with the technical difficulty that the solution $\mathbf v$ for the target incompressible Euler system is not an admissible test function in the relative entropy inequality. Motivated by \cite{e6}, we introduce a corrector in the relative energy inequality to control the convergence of terms involving $\epsilon, M$. Compared with \cite{e6}, our principal obstacle is the convergence of the viscous term, whereas \cite{e6} considered the constant viscosity case. To this end, we follow \cite{bi2}, who gave a pure relative energy argument which does not rely on the Bresch-Desjardins structure. In this sense, our proof is conceptually simple, {in comparison with~\cite{n}}. Moreover, we treat some parts of the remainder $\mathcal{R}$ in an alternative and arguably simpler fashion. 

It should not come as a surprise that degeneracy of the viscosity, while posing substantial challenges to the existence theory, does not impede the applicability of relative energy arguments. In fact, the latter works particularly well for the Euler equations themselves, which somehow display the maximally degenerate viscosity (viz.~zero).

Our assumptions on~\eqref{1.1} with shear viscosity $\rho$ and bulk viscosity $0$ are chosen, for simplicity, as the prototypic choice satisfying the Bresch-Desjardins relation; this choice and the inclusion of the drag term play, for our purposes, the mere role of ensuring the existence of weak solutions of~\eqref{1.1} for any fixed $\epsilon, M$ thanks to~\cite{b1}. One can easily imagine more general weak-strong convergence results under the conditional hypothesis that the approximate problems have a solution. Similarly, working on the whole space instead of expanding domains, one could remove the drag term, which is not required for existence of~\eqref{1.1} on the whole space.

In Section 2, we introduce the definition of weak solution and some necessary useful lemmas, then we state the main theorem. The key ingredient of the proof is the relative entropy inequality stated in Section 3. Section 4 is devoted to the proof of the main theorem.

\section{Preliminaries} \label{S2}

\subsection{Definition of weak solutions }


\begin{definition}\label{def1}
We say that $(\rho,\mathbf u)$ is a weak solution to~\eqref{1.1} {with initial data $(\rho_0,\mathbf u_0)$  if the following conditions are satisfied:
\begin{align}\label{2.1}
&\rho\geq0,\hspace{3pt} \rho|\mathbf u|^2\in L^\infty(0,T; L^1(\Omega_M)),
\hspace{3pt}\rho\in L^\infty(0,T;L^\gamma(\Omega_M)),
\nabla_x\sqrt{\rho}\in L^\infty(0,T;L^2(\Omega_M));\\
&\rho\mathbf u|_{\partial\Omega_M}=0\hspace{3pt}\text{in}\hspace{3pt}L^2((0,T);L^1(\partial\Omega_M)),\hspace{5pt}\nabla_x\rho\times n|_{\partial\Omega_M}=0,\text{in}\hspace{3pt}L^2((0,T);L^\infty(\partial\Omega_M));\nonumber
\end{align} }

\noindent
$\bullet$ the continuity equation
\begin{align}\label{2.2}
\int_{\Omega_M}\rho_0\varphi(0) dx+\int^T_0\int_{\Omega_M}\rho\partial_t\varphi+\rho\mathbf{u}\cdot\nabla_x\varphi dxdt=0
\end{align}
holds for all $\varphi\in C^\infty_c([0,T)\times\Omega)$;

\noindent
$\bullet$
the momentum equation
\begin{align}\label{2.3}
\int_{\Omega_M}\rho_0\mathbf u_0\varphi(0) dx&+\int^T_0\int_{\Omega_M}\rho\mathbf{u}\partial_t\varphi+ \sqrt{\rho}\mathbf{u}\otimes\sqrt{\rho}\mathbf{u}:\nabla_x\varphi
+\frac{1}{\varepsilon^2}p(\rho)\operatorname{div}_x\varphi dxdt\nonumber\\
&\hspace{8pt}-\varepsilon^\alpha\int^T_0\int_{\Omega_M}\sqrt{\rho}\mathbb{S}_\mu:\nabla_x\varphi dxdt
-r_1(\varepsilon)\int^T_0\int_{\Omega_M}\sqrt{\rho}|\mathbf u|\sqrt{\rho}\mathbf u\cdot\varphi dxdt=0
\end{align}
holds for all $\varphi\in C^\infty_c([0,T)\times\Omega)$; 


\noindent
$\bullet$
the energy inequality
\begin{align}\label{2.4}
\int_{\Omega_M}\frac{1}{2}\rho|\mathbf{u}|^2(\tau)+\frac{1}{\varepsilon^2}H(\rho(\tau))dx
+\varepsilon^\alpha\int^\tau_0\int_{\Omega_M}|\mathbb{S}_\mu|^2 dxdt
+r_1(\varepsilon)\int^\tau_0\int_{\Omega_M}\rho|\mathbf u|^3dxdt\nonumber\\
\leq\int_{\Omega_M}\frac{1}{2}\rho_0|\mathbf{u}_0|^2+\frac{1}{\varepsilon^2}H(\rho_0)dx
\end{align}
holds for a.e.\ $\tau\in(0,T)$,
where $H(\rho)=\rho\int^\rho_1\frac{p(z)}{z^2}dz$.
{
\begin{remark}
Here we adopt the idea from Lacroix-Violet and Vasseur \cite{la} and of Bresch et al.~\cite{b3} to rewrite the viscous term as
\begin{align*}
\text{div}_{\bf x}(\rho D_x(\mathbf{u}))=\text{div}_{\bf x}(\sqrt{\rho}\mathbb{S}_{\mu})
\end{align*}
where $\mathbb{S}_{\mu}$ is the symmetric part of $\mathbb{T}_\mu$ as
\begin{align*}
\sqrt{\rho}\mathbb{T}_{\mu}=\nabla_x(\rho\mathbf u)-\sqrt{\rho}\mathbf u\otimes\sqrt{\rho}\nabla_x\ln\rho,
\hspace{3pt}\text{in}\hspace{3pt}\mathcal{D}'.
\end{align*}

\end{remark}
}

\end{definition}
This definition is the one from~\cite{b1}. There, it is shown that the boundary conditions in~\eqref{2.1} actually make sense, that the energy inequality follows automatically (so we included it in the definition merely for convenience), and that weak solutions exist for any data that satisfy~\eqref{2.1} initially.

\subsection{Energy bounds}
Here, we use the decomposition
\begin{align*}
h =[h]_{ess} +[h]_{res},\hspace{10pt}[h]_{ess}=\chi(\rho )h,\hspace{5pt}[h]_{res}=(1-\chi(\rho))h,
\end{align*}
where
\begin{align*}
\chi \in C_c^\infty(0,\infty),\hspace{5pt}0\leq\chi(\rho)\leq1,\hspace{5pt}\chi(\rho)=1\hspace{3pt}\text{in a neighbourhood of}\hspace{3pt}\rho=1.
\end{align*}
Then we have, as in~\cite{e5},
\begin{align}
&\operatorname{esssup}_{t\in[0,T]}\|\sqrt{\rho}\mathbf u(t,\cdot)\|_{L^2(\Omega_M)}\leq C,\hspace{3pt}\|\mathbb{S}_\mu\|_{L^2((0,T),L^2(\Omega_M))}\leq \frac{C}{\varepsilon^{\frac{\alpha}{2}}},\label{2.5}\\
&\operatorname{esssup}_{t\in[0,T]}\big{(}\|1_{res}\|_{L^1(\Omega_M)}+\|[\rho^\gamma]_{res}\|_{L^1(\Omega_M)}\big{)}\leq \varepsilon^2C,\label{2.6}\\
&\operatorname{esssup}_{t\in[0,T]}\left\|\left[\frac{\rho-1}{\varepsilon}\right]_{ess}\right\|_{L^2(\Omega_M)}\leq C.\label{2.7}
\end{align}

\subsection{Solutions of the target system}
We suppose the initial velocity $\mathbf u_0$ in \eqref{1.4} is sufficiently smooth and confined to a compact set, then the inviscid system \eqref{1.8} is known to possess a smooth solution at least on a short time interval, see \cite{ka}. More precisely, if
\begin{align*}
&\mathbf u_0\in C_c^m(\mathbb{R}^3)\hspace{3pt} \text{for some}\hspace{3pt}m>4, \\
&\mathbf v(0,\cdot)=\mathbf v_0=H[\mathbf u_0],
\end{align*}
then the system \eqref{1.8} possesses a solution
\begin{align}
\mathbf v\in C^k([0,T_{\max});W^{m-k,2}(\mathbb{R}^3)),\quad k=1,2...m-1,
\end{align}
where $T_{\max}>0$. Here, for a vector field $\mathbf u\in L^2$, we denote by $H[\mathbf u]$ the standard Helmholtz projection onto the space of solenoidal functions.

\subsection{Acoustic waves}
Rewriting the Navier-Stokes equations as

\begin{equation}\label{2.9}
\left\{
\begin{array}{llll}
&\varepsilon\partial_t\frac{\rho-1}{\varepsilon}+\text{div}_{x}(\rho\mathbf{u})=0,\\
&\varepsilon\partial_{t}(\rho\mathbf u)+p'(1)\nabla_{x}\left(\frac{\rho-1}{\varepsilon}\right)
=\varepsilon\big{(}\varepsilon^\alpha\text{div}_{x}(\sqrt{\rho}\mathbb{S}_{\mu})
-\text{div}_{x}(\rho\mathbf{u}\otimes\mathbf{u})\\
&\hspace{3cm}-\frac{1}{\varepsilon^2}\nabla_x(p(\rho)-p'(1)(\rho-1)-p(1))
-r_1(\varepsilon)|\mathbf u|\mathbf u\big{)},
\end{array}\right.
\end{equation}
motivates the corresponding acoustic system:
\begin{equation}\label{2.10}
\left\{
\begin{array}{llll}
\varepsilon\partial_ts+\Delta_x\Phi=0\\
\varepsilon\partial_{t}\nabla_x\Phi+p'(1)\nabla_xs=0,
\end{array}\right.
\end{equation}
with initial data:
\begin{align*}
s(t=0,\cdot)=\rho^{(1)}_0,\hspace{5pt}\nabla_x\Phi(t=0,\cdot)=\nabla_x\Phi_0=\mathbf u_0-H[\mathbf u_0].
\end{align*}
From \cite{s}, we have the conservation laws
\begin{align}\label{2.11}
\left(\|\nabla_x\Phi(t,\cdot)\|_{W^{k,2}(\mathbb{R}^3)}+\|s(t,\cdot)\|_{W^{k,2}(\mathbb{R}^3)}\right)|^{t=\tau}_{t=0}=0,
\hspace{3pt}k=0,1,\ldots
\end{align}
for any $\tau>0$, as well as the
{Strichartz estimates
\begin{align}\label{2.12}
\|\nabla_x\Phi(\tau,\cdot)\|_{W^{k,p}(\mathbb{R}^3)}+\|s(\tau,\cdot)\|_{W^{k,p}(\mathbb{R}^3)}\leq C\left(1+\frac{\tau}{\varepsilon}\right)^{-(\frac{1}{q}-\frac{1}{p})}
\big{(}\|\nabla_x\Phi_0\|_{W^{k+d,q}(\mathbb{R}^3)}+\|\rho^{(1)}_0\|_{W^{k+d,q}(\mathbb{R}^3)}\big{)},
\end{align}
where $2\leq p\leq\infty$, $\frac{1}{p}+\frac{1}{q}=1$, $k=0,1,\ldots$, and $d>3\left(\frac{1}{q}-\frac{1}{p}\right)$.
}


\begin{remark}\label{r2.1}
From to the compact support assumption on the initial data of \eqref{2.15} {and the hyperbolicity of~\eqref{2.10}}, we deduce that, for all $t\geq0$,
\begin{align*}
s(t,x)=0,\hspace{3pt}\Delta_x\Phi(t,x)=0\hspace{3pt}\text{for}\hspace{3pt}|x|\geq c+\sqrt{p'(1)}\frac{t}{\varepsilon}.
\end{align*}
Accordingly, {due to hypothesis (H4)}, the waves emanating from a compact subset $K$ of $\Omega_M$ can never reach the boundary within a fixed timespan $(0,T)$, for sufficiently small $\epsilon$. This means that the behaviour of solutions of the acoustic system on
$$K(t)=\left\{(t,x):\ 0\le t\le T\ \text{and }|x-x_0|\le \sqrt{p'(1)}\frac{t}{\varepsilon}\ \text{for some }x_0\in K\right\}$$
is the same as in the expanding domain $\Omega_M$ ({see Fig.1}). The expanding domain setup thus excludes the possibility of reflection of acoustic waves at the boundary.  
\end{remark}

\begin{figure}[h]
\centering
\includegraphics[scale=0.26]{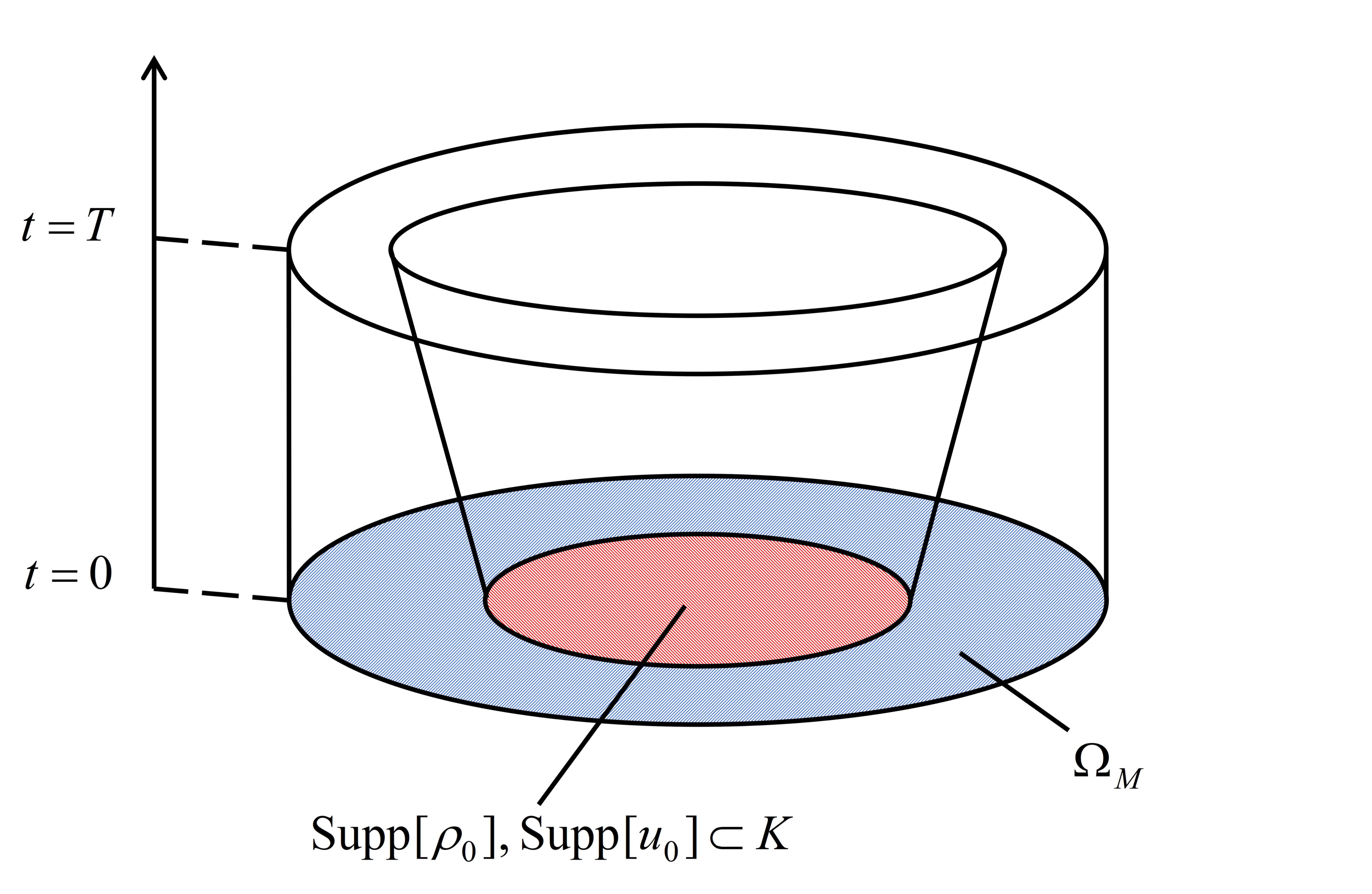}
\caption{$K(t)\subset\subset \Omega_M$}
\end{figure}
\vskip 0.5cm

\subsection{Cut-off operators}
From Remark \ref{r2.1}, we deduce for given $T$ and sufficiently small $\epsilon$ the identity
\begin{align*}
\nabla_x\Phi|_{\partial\Omega_M}=\nabla_x\Phi_0|_{\partial\Omega_M},\hspace{3pt}\text{for all}\hspace{3pt}t\in(0,T).
\end{align*}
Following \cite{e6}, we introduce the cut-off function $w_M=-\eta_M\mathbf v-\eta_M\nabla_x\Phi$,
where $\eta_M$ satisfies
\begin{align*}
\eta_M\in C^\infty_c(\mathbb{R}^3),\hspace{3pt}0\leq\eta_M\leq1,\hspace{3pt}
\eta_M|_{\partial\Omega_M}=1,\hspace{3pt}\eta_M(x)=0,\hspace{3pt}\operatorname{dist}(x,\partial\Omega_M)>1,
\end{align*}
and we have the estimate
\begin{align}\label{2.13}
\|\partial_tw_M(\tau,\cdot)\|_{L^{p}(\Omega_M)}+\|w_M(\tau,\cdot)\|_{W^{2,p}(\Omega_M)}
\leq CM^{2\left(\frac{1}{p}-1\right)}
\end{align}
for $1\leq p\leq\infty$, $\tau\in(0,T)$.

\subsection{Main result}
\begin{theorem}\label{thm2.1}
Let $\alpha>0$ and suppose the initial data \eqref{1.4} are chosen in such a way that
\begin{align}
&\mathbf v_0\in C^m(\mathbb{R}^3),\hspace{2pt}\rho^{(1)}_0\in C^m(\mathbb{R}^3),\hspace{2pt} m>4,\\
&\operatorname{supp}[\mathbf u_0],\hspace{2pt}\operatorname{supp}[\rho^{(1)}_0]\hspace{2pt}\text{compact in}\hspace{2pt}\mathbb{R}^3.\label{2.15}
\end{align}
Let $T_{\max}>0$ be the life-span of the smooth solution $\mathbf v$ of the Euler system \eqref{1.8}. Let $0<T<T_{\max}$, then for any weak solutions $(\rho,\mathbf u)$ of the Navier-Stokes equations \eqref{1.1}, under the stated hypotheses (H1)--(H4) we have
\begin{align*}
&\sqrt{\rho}\mathbf u\rightarrow\mathbf v\hspace{5pt}\text{strongly in}\hspace{3pt}L^{2}(0,T;L^2(K)),\\
&\rho\rightarrow1\hspace{5pt}\text{strongly in}\hspace{3pt}L^\infty(0,T;L^{\min\{2,\gamma\}}(K)),
\end{align*}
for any compact $K\subset\mathbb{R}^3$.
\end{theorem}
\begin{remark}
The result remains valid with well-prepared initial data. We omit the details of the minor modifications.
\end{remark}

\begin{remark}
Adding to~\cite{f,n}, our result exhibits another case where the limit system does not depend on whether a constant or a variable density approximation is chosen. 


\end{remark}


\vskip 0.5cm

\section{Relative entropy inequality and main theorem}

Motivated by \cite{e2,e3}, for any finite weak solution $(\rho,\mathbf u)$ to the system \eqref{1.1}, we introduce the relative energy functional
\begin{align}\label{3.1}
\mathcal{E}(\rho,\mathbf{u}|r, \mathbf{U})&=\int_{\Omega_M}\left[\frac{1}{2}\rho|\mathbf u-\mathbf U|^2
+\frac{1}{\varepsilon^2}\big{(}H(\rho)-H'(r)(\rho-r)-H(r)\big{)}\right]dx,
\end{align}
where $r>0$, $\mathbf U|$ are smooth `test' functions compactly supported in $\Omega_M$. Later, we will make the choice $r=1+\varepsilon s$ and $\mathbf U=\mathbf v+\nabla_x\Phi+w_M$.

\begin{lemma}\label{relativeentropy}
Let $(\rho,\mathbf{u})$ be a weak solution introduced in Definition  \ref{def1}. Then $(\rho,\mathbf{u})$
satisfy the relative entropy inequality

\begin{align}
\mathcal{E}(\rho,\mathbf{u}|r,\mathbf U)|^{t=\tau}_{t=0}+\varepsilon^\alpha\int^\tau_0\int_{\Omega_M}|\mathbb{S}_\mu|^2~dxdt
\leq\int^\tau_0\mathcal{R}(\rho,\mathbf u,r,\mathbf U)dt,
\end{align}
where
\begin{align}\label{3.3}
\mathcal{R}(\rho,\mathbf u,r,\mathbf U)
=&\int_{\Omega_M}\sqrt{\rho}(\sqrt{\rho}\mathbf U-\sqrt{\rho}\mathbf u)\cdot(\partial_t\mathbf U+\mathbf u\cdot\nabla_x\mathbf U)~dx\nonumber\\
&+\frac{1}{\varepsilon^2}\int_{\Omega_M}\rho(\mathbf U-\mathbf u)\cdot\nabla_x(H'(r)-H''(1)(r-1)-H'(1))~dx\nonumber\\
&+\frac{1}{\varepsilon^2}\int_{\Omega_M}\big{(}p(r)-p'(r)(r-\rho)-p(\rho))\operatorname{div}_x\mathbf U~dx\nonumber\\
&+\frac{1}{\varepsilon^2}\int_{\Omega_M}(r-\rho)H''(r)\operatorname{div}_xw_M~dx
+\frac{1}{\varepsilon}\int_{\Omega_M}(r-\rho)H''(r)\operatorname{div}_x(s\mathbf U)~dx\nonumber\\
&-\int_{\Omega_M}\rho\partial_t\nabla_x\Phi\cdot(\mathbf U-\mathbf u)~dx+\varepsilon^\alpha\int^\tau_0\int_{\Omega_M}\sqrt{\rho}\mathbb{S}_\mu:\nabla_x\mathbf U~dxdt
\nonumber\\
&+r_1(\varepsilon)\int^\tau_0\int_{\Omega_M}|\sqrt{\rho}\mathbf u|\sqrt{\rho}\mathbf u\cdot\mathbf U~dxdt.
\end{align}
\end{lemma}
{\bf Proof:}
Using the relative energy~\eqref{3.1} and energy inequality \eqref{2.4}, we have
\begin{align}\label{3.4}
\mathcal{E}(\rho,\mathbf{u}|r, \mathbf{U})|^{t=\tau}_{t=0}
\leq&\int_{\Omega_M}(\frac{1}{2}\rho|\mathbf U|^2-\rho\mathbf u\cdot\mathbf U)|^{t=\tau}_{t=0}~dx
-\frac{1}{\varepsilon^2}\int_{\Omega_M}\big{(}H(r)+H'(r)(\rho-r)\big{)}|^{t=\tau}_{t=0}~dx\nonumber\\
&-\varepsilon^\alpha\int^\tau_0\int_{\Omega_M}\sqrt{\rho}|\mathbb{S}_\mu|^2~dxdt.
\end{align}

From the weak formulation \eqref{2.2}--\eqref{2.3}, we deduce
\begin{align}\label{3.7}
&\int_{\Omega_M}\frac{1}{2}\rho|\mathbf U|^2|^{t=\tau}_{t=0}~dx=\int^\tau_0\int_{\Omega_M}\rho\mathbf U\cdot\partial_t\mathbf U~dxdt+\int^\tau_0\int_{\Omega_M}\rho\mathbf u(\mathbf U\cdot\nabla_x\mathbf U)~dxdt,\\
&-\int_{\Omega_M}\rho\mathbf u\cdot\mathbf U|^{t=\tau}_{t=0}~dx=-\int^\tau_0\int_{\Omega_M}\left(\rho\mathbf u\cdot\partial_t\mathbf U+\rho\mathbf u\otimes\mathbf u:\nabla_x\mathbf U\right)~dxdt-\frac{1}{\varepsilon^2}\int^\tau_0\int_{\Omega_M}p(\rho)\text{div}_x\mathbf U~dxdt\nonumber\\
&\hspace{80pt}+\varepsilon^\alpha\int^\tau_0\int_{\Omega_M}\sqrt{\rho}\mathbb{S}_\mu:\nabla_x\mathbf U~dxdt
+r_1(\varepsilon)\int^\tau_0\int_{\Omega_M}|\sqrt{\rho}\mathbf u|\sqrt{\rho}\mathbf u\cdot\mathbf U~dxdt,\\
&-\frac{1}{\varepsilon^2}\int_{\Omega_M}\big{(}H(r)+H'(r)(\rho-r)\big{)}|^{t=\tau}_{t=0}~dx
=\frac{1}{\varepsilon^2}\int^\tau_0\int_{\Omega_M}[\partial_t(H'(r)(r-\rho)-\rho\mathbf u\cdot\nabla_x(H'(r))]~dxdt.
\end{align}

Summing \eqref{3.4}-\eqref{3.7} together, we obtain
\begin{align*}
\mathcal{E}(\rho,\mathbf{u}|r,\mathbf U)|^{t=\tau}_{t=0}\leq\int^\tau_0\int_{\Omega}\mathcal{R}(\rho,\mathbf u,r,\mathbf U)~dxdt
-\varepsilon^\alpha\int^\tau_0\int_{\Omega_M}\sqrt{\rho}|\mathbb{S}_\mu|^2~dxdt,
\end{align*}
where
\begin{align}\label{3.8}
\mathcal{R}(\rho,\mathbf u,r,\mathbf U)
=&\int_{\Omega_M}\rho(\mathbf U-\mathbf u)\cdot\partial_t\mathbf U~dx+\int_{\Omega_M}\sqrt{\rho}\mathbf u(\sqrt{\rho}\mathbf U-\sqrt{\rho}\mathbf u)\cdot\nabla_x\mathbf U~dx\nonumber\\
&
+\frac{1}{\varepsilon^2}\int_{\Omega_M}[(r-\rho)\partial_tH'(r)-\rho\mathbf u\cdot\nabla_xH'(r)]~dx
-\frac{1}{\varepsilon^2}\int_{\Omega_M}p(\rho)\text{div}_x\mathbf U~dx\nonumber\\
&\hspace{10pt}
+\varepsilon^\alpha\int^\tau_0\int_{\Omega_M}\sqrt{\rho}\mathbb{S}_\mu:\nabla_x\mathbf U~dxdt
+r_1(\varepsilon)\int^\tau_0\int_{\Omega_M}|\sqrt{\rho}\mathbf u|\sqrt{\rho}\mathbf u\cdot\mathbf U~dxdt\nonumber\\
=&\int_{\Omega_M}\sqrt{\rho}(\sqrt{\rho}\mathbf U-\sqrt{\rho}\mathbf u)\cdot(\partial_t\mathbf U+\mathbf u\cdot\nabla_x\mathbf U)~dxdt\nonumber\\
&+\frac{1}{\varepsilon^2}\int_{\Omega_M}[(r-\rho)\partial_tH'(r)+\nabla_xH'(r)(r\mathbf U-\rho\mathbf u)]~dx
-\frac{1}{\varepsilon^2}\int_{\Omega_M}(p(\rho)-p(r))\text{div}_x\mathbf U~dx\nonumber\\
&+\varepsilon^\alpha\int^\tau_0\int_{\Omega_M}\sqrt{\rho}\mathbb{S}_\mu:\nabla_x\mathbf U~dxdt
+r_1(\varepsilon)\int^\tau_0\int_{\Omega_M}|\sqrt{\rho}\mathbf u|\sqrt{\rho}\mathbf u\cdot\mathbf U~dxdt.
\end{align}

It is easy to calculate
\begin{align}\label{3.9}
(r-\rho)&\partial_tH'(r)+\nabla_xH'(r)(r\mathbf U-\rho\mathbf u)-(p(\rho)-p(r))\text{div}_x\mathbf U\nonumber\\
&=\big{(}p(r)-p'(r)(r-\rho)-p(\rho)\big{)}\text{div}_x\mathbf U
+(r-\rho)\partial_tH'(r)\nonumber\\
&\hspace{10pt}+(r-\rho)p'(r)\text{div}_x\mathbf U
+(r-\rho)\nabla_xH'(r)\cdot U
+\rho\nabla_xH'(r)(\mathbf U-\mathbf u)\nonumber\\
&=\big{(}p(r)-p'(r)(r-\rho)-p(\rho)\big{)}\text{div}_x\mathbf U
+(r-\rho)H''(r)(\partial_tr+\text{div}_x(r\mathbf U))
+\rho\nabla_xH'(r)(\mathbf U-\mathbf u).
\end{align}
Setting now $r=1+\varepsilon s$, $\mathbf U=\mathbf v+\nabla_x\Phi+w_M$ and using $\operatorname{div}_x\mathbf v=0$, we have
\begin{align}\label{3.10}
\partial_tr+\text{div}_x(r\mathbf U)&=\partial_tr+\text{div}_x[(r-\varepsilon s)\mathbf U]+\varepsilon\text{div}_x(s\mathbf U)\nonumber\\
&=\varepsilon\partial_ts+\text{div}_x\mathbf v+\text{div}_x(\nabla_x\Phi)+\text{div}_xw_M+\varepsilon\text{div}_x(s\mathbf U)\nonumber\\
&=\text{div}_xw_M+\varepsilon\text{div}_x(s\mathbf U).
\end{align}
On the other hand, it is straightforward that
\begin{align}\label{3.11}
\rho\nabla_xH'(r)(\mathbf U-\mathbf u)=\rho\nabla_x\big{(}H'(r)&-H''(1)(r-1)-H'(1)\big{)}\cdot(\mathbf U-\mathbf u)\\
&+\rho\nabla_x(H''(1)(r-1))\cdot(\mathbf U-\mathbf u),
\end{align}
while we obtain, recalling the acoustic system \eqref{2.10},
\begin{align}\label{3.12}
\rho\nabla_x(H''(1)(r-1))\cdot(\mathbf U-\mathbf u)&=\varepsilon\rho\nabla_x(H''(1)s)\cdot(\mathbf U-\mathbf u)\nonumber\\
&=\varepsilon\nabla_x(p'(1)s)\cdot(\mathbf U-\mathbf u)=-\varepsilon^2\rho\partial_t\nabla_x\Phi\cdot(\mathbf U-\mathbf u).
\end{align}
Plugging \eqref{3.9}--\eqref{3.12} into \eqref{3.8}, we get, as desired,
\begin{align*}
\mathcal{R}(\rho,\mathbf u,r,\mathbf U)
=&\int_{\Omega_M}\sqrt{\rho}(\sqrt{\rho}\mathbf U-\sqrt{\rho}\mathbf u)\cdot(\partial_t\mathbf U+\mathbf u\cdot\nabla_x\mathbf U)~dx\\
&+\frac{1}{\varepsilon^2}\int_{\Omega_M}\rho(\mathbf U-\mathbf u)\cdot\nabla_x(H'(r)-H''(1)(r-1)-H'(1))~dx\\
&+\frac{1}{\varepsilon^2}\int_{\Omega_M}\big{(}p(r)-p'(r)(r-\rho)-p(\rho))\text{div}_x\mathbf U~dx+\frac{1}{\varepsilon^2}\int_{\Omega_M}(r-\rho)H''(r)\text{div}_xw_M~dx\\
&
+\frac{1}{\varepsilon}\int_{\Omega_M}(r-\rho)H''(r)\text{div}_x(s\mathbf U)~dx
-\int_{\Omega_M}\rho\partial_t\nabla_x\Phi\cdot(\mathbf U-\mathbf u)~dx\\
&+\varepsilon^\alpha\int^\tau_0\int_{\Omega_M}\sqrt{\rho}\mathbb{S}_\mu:\nabla_x\mathbf U~dxdt
+r_1(\varepsilon)\int^\tau_0\int_{\Omega_M}|\sqrt{\rho}\mathbf u|\sqrt{\rho}\mathbf u\cdot\mathbf U~dxdt.
\end{align*}
\qed

\section{Proof of Theorem \ref{thm2.1}}
In this section, which constitutes the essential part of the proof of the main result, we will show that all the integrals in the remainder term $\mathcal{R}$ on the right-hand side of \eqref{3.3} can be absorbed into the left-hand side, which will enable us to conclude by Gr\"onwall's inequality. First, recalling $r=1+\varepsilon s$ and $\mathbf U=\mathbf v+\nabla_x\Phi+w_M$, we rewrite \eqref{3.3} as
\begin{align}\label{4.1}
\mathcal{R}(\rho,\mathbf u,r,\mathbf U)
=&\int_{\Omega_M}\sqrt{\rho}(\sqrt{\rho}\mathbf U-\sqrt{\rho}\mathbf u)\cdot(\partial_t\mathbf v+\partial_tw_M+\mathbf u\cdot\nabla_x\mathbf U)~dx\nonumber\\
&+\frac{1}{\varepsilon^2}\int_{\Omega_M}\rho(\mathbf U-\mathbf u)\cdot\nabla_x(H'(r)-H''(1)(r-1)-H'(1))~dx\nonumber\\
&+\frac{1}{\varepsilon^2}\int_{\Omega_M}\big{(}p(r)-p'(r)(r-\rho)-p(\rho))\text{div}_x\mathbf U~dx\nonumber\\
&+\frac{1}{\varepsilon^2}\int_{\Omega_M}(r-\rho)H''(r)\text{div}_xw_M
~dx+\frac{1}{\varepsilon}\int_{\Omega_M}(r-\rho)H''(r)\text{div}_x(s\mathbf U)~dx\nonumber\\
&+\varepsilon^\alpha\int^\tau_0\int_{\Omega_M}\sqrt{\rho}\mathbb{S}_\mu:\nabla_x\mathbf U~dxdt
+r_1(\varepsilon)\int^\tau_0\int_{\Omega_M}|\sqrt{\rho}\mathbf u|\sqrt{\rho}\mathbf u\cdot\mathbf U~dxdt\nonumber\\
=&\sum^{7}_{i=1}R_i.
\end{align}
First, we deal with term $R_1$ as:
\begin{align}\label{4.2}
&\int^\tau_0\int_{\Omega_M}\sqrt{\rho}(\sqrt{\rho}\mathbf U-\sqrt{\rho}\mathbf u)\cdot(\partial_t\mathbf v+\partial_tw_M+\mathbf u\cdot\nabla_x\mathbf U)~dxdt\nonumber\\
=&\int^\tau_0\int_{\Omega_M}\sqrt{\rho}(\sqrt{\rho}\mathbf U-\sqrt{\rho}\mathbf u)\cdot(\partial_t\mathbf v+\partial_tw_M+\mathbf U\cdot\nabla_x\mathbf U)~dxdt\nonumber\\
&\hspace{20pt}-\int^\tau_0\int_{\Omega_M}\rho(\mathbf U-\mathbf u)\otimes(\mathbf U-\mathbf u):\nabla_x\mathbf U~dxdt,
\end{align}
where
\begin{align}
\left|\int^\tau_0\int_{\Omega_M}\rho(\mathbf U-\mathbf u)\otimes(\mathbf U-\mathbf u):\nabla_x\mathbf U~dxdt\right|\leq2\int^\tau_0\left\|\nabla_x\mathbf v+\nabla_xw_M+\nabla^2_x\Phi\right\|_{L^\infty(\Omega_M)}\mathcal{E}dt.
\end{align}
We manipulate the first term on the right side of \eqref{4.2} as
\begin{align}\label{4.4}
&\int^\tau_0\int_{\Omega_M}\sqrt{\rho}(\sqrt{\rho}\mathbf U-\sqrt{\rho}\mathbf u)\cdot(\partial_t\mathbf v+\partial_tw_M+\mathbf U\cdot\nabla_x\mathbf U)~dxdt\nonumber\\
=&\int^\tau_0\int_{\Omega_M}\rho(\mathbf U-\mathbf u)\cdot(\partial_t\mathbf v+\partial_tw_M+\mathbf U\cdot\nabla_x\mathbf U)~dxdt\nonumber\\
=&\int^\tau_0\int_{\Omega_M}\rho(\mathbf U-\mathbf u)\cdot(\partial_t\mathbf v+\mathbf v\cdot\nabla_x\mathbf v)~dxdt
+\int^\tau_0\int_{\Omega_M}\rho(\mathbf U-\mathbf u)\partial_tw_M~dxdt\nonumber\\
&+\int^\tau_0\int_{\Omega_M}\rho(\mathbf U-\mathbf u)\otimes(\nabla_x\Phi+w_M):\nabla_x\mathbf v~dxdt
+\int^\tau_0\int_{\Omega_M}\rho(\mathbf U-\mathbf u)\frac{1}{2}\nabla_x|\nabla_x\Phi+w_M|^2~dxdt\nonumber\\
&+\int^\tau_0\int_{\Omega_M}\rho(\mathbf U-\mathbf u)\otimes\mathbf v:(\nabla_x^2\Phi+\nabla_xw_M)~dxdt\nonumber\\
=&\sum^5_{j=1}R_{1j}.
\end{align}
By virtue of \eqref{2.6}--\eqref{2.7} and \eqref{2.12}--\eqref{2.13}, we can control $R_{13}$ as 
{
\begin{align}\label{4.15}
|&\int^\tau_0\int_{\Omega_M}\rho(\mathbf U-\mathbf u)\otimes(\nabla_x\Phi+w_M):\nabla_x\mathbf v\nonumber~dxdt|\nonumber\\
\le&
C\int_0^\tau\int_{\Omega_M}
\rho|\mathbf U-\mathbf u|^2 dx dt
+\int_0^\tau\int_{\Omega_M}|[\rho]_{ess}+[\rho]_{res}||\nabla_x\Phi+w_M|^2|\nabla_x\mathbf v|^2dxdt\nonumber\\
\le &
C\int_0^\tau\int_{\Omega_M}
\rho|\mathbf U-\mathbf u|^2 dx dt+\int_0^\tau\int_{\Omega_M}\left|\varepsilon\left[\frac{\rho-1}{\varepsilon}\right]_{ess}+1+[\rho]_{res}\right|\nabla_x\Phi+w_M|^2|\nabla_x\mathbf v|^2dxdt\nonumber\\
\le &
C\int_0^\tau\int_{\Omega_M}
\rho|\mathbf U-\mathbf u|^2 dx dt\nonumber\\
&+\varepsilon\int_0^\tau\left\|\left[\frac{\rho-1}{\varepsilon}\right]_{ess}\right\|_{L^2(\Omega_M)}\|\nabla_x\Phi+w_M\|^2_{L^\infty(\Omega_M)}\|\nabla_x\mathbf v\|_{L^4(\Omega_M)}^2 dt\nonumber\\
&+\int_0^\tau\|\nabla_x\Phi
+w_M\|_{L^\infty(\Omega_M)}^2
\|\nabla_x\mathbf v\|_{L^2(\Omega_M)}^2dt\nonumber\\
&+\int_0^\tau  \|[\rho]_{res}\|_{L^{\gamma}(\Omega_M)}
\|\nabla_x\Phi+w_M\|^2_{L^\infty(\Omega_M)}
\|\nabla_x\mathbf v\|_{L^{\frac{2\gamma}{\gamma-1}}(\Omega_M)}^2dt.
\end{align}
By Sobolev's inequality, for any $p>3$ and some $C=C(p)$, we get the estimate
$$
\|\nabla_x\Phi+w_M\|_{L^\infty(\Omega_M)}
\leq C\left(\|\nabla_x\Phi\|_{W^{1,p}(\Omega_M)}+\|w_M\|_{W^{1,p}(\Omega_M)}\right)\le C\left[\left(1+\frac{t}{\varepsilon}\right)^{-\left(1-\frac{2}{p}\right)}
+M^{2\left(\frac{1}{p}-1\right)}\right].
$$
Thus we can control \eqref{4.15} as
\begin{align*}
&\int^\tau_0\int_{\Omega_M}\rho(\mathbf U-\mathbf u)\otimes(\nabla_x\Phi+w_M):\nabla_x\mathbf v\nonumber~dxdt\le
C\int_0^\tau\int_{\Omega_M}
\rho|\mathbf U-\mathbf u|^2 dx dt+C\left(\varepsilon^{2-\frac{4}{p}}+ M^{4\left(\frac{1}{p}-1\right)}
\right).
\end{align*}
}
In a similar way we control the terms $\int^\tau_0\int_{\Omega_M}\rho(\mathbf U-\mathbf u)\frac{1}{2}\nabla_x|\nabla_x\Phi+w_M|^2dx$ and $\int^\tau_0\int_{\Omega_M}\rho(\mathbf U-\mathbf u)\otimes\mathbf v:(\nabla_x^2\Phi+\nabla_xw_M)dx$. 

Next, {we turn to $R_{11}$: 
\begin{align}\label{4.16}
&\int^\tau_0\int_{\Omega_M}\rho(\mathbf U-\mathbf u)\cdot(\partial_t\mathbf v+\mathbf v\cdot\nabla_x\mathbf v)~dxdt \nonumber\\
=&-\int^\tau_0\int_{\Omega_M}\rho(\mathbf U-\mathbf u)\cdot\nabla_x\Pi~dxdt \nonumber\\
=&\int^\tau_0\int_{\Omega_M}\rho\mathbf u\cdot\nabla_x\Pi~dxdt -\int^\tau_0\int_{\Omega_M}\rho\mathbf U\cdot\nabla_x\Pi~dxdt.
\end{align}
Now we follow Feireisl et al. \cite{e6}. For the readers' convenience, we recall the calculation as
\begin{align}
\int^\tau_0\int_{\Omega_M}\rho\mathbf u\cdot\nabla_x\Pi~dxdt&=
-\int^\tau_0\int_{\Omega_M}\rho\partial_t\Pi~dxdt+\left[\int_{\Omega_M}\rho\Pi dx\right]^{t=\tau}_{t=0}\nonumber\\
&=-\varepsilon\int^\tau_0\int_{\Omega_M}\frac{\rho-1}{\varepsilon}\partial_t\Pi dx+\varepsilon\left[\int_{\Omega_M}\frac{\rho-1}{\varepsilon}\Pi dx\right]^{t=\tau}_{t=0}.
\end{align}
Using \eqref{2.6}--\eqref{2.7}, we see that these terms converge to zero. Now we turn to the second term on the right side of \eqref{4.16} as  
\begin{align}
\left|\int^\tau_0\int_{\Omega_M}\rho\mathbf U\cdot\nabla_x\Pi~dxdt\right|
&\leq\left|\varepsilon\int^\tau_0\int_{\Omega_M}\frac{\rho-1}{\varepsilon}\mathbf U\cdot\nabla_x\Pi~dxdt\right|
+\left|\int^\tau_0\int_{\Omega_M}\mathbf U\cdot\nabla_x\Pi~dxdt\right|\nonumber\\
&\leq C\varepsilon+\left|\int^\tau_0\int_{\Omega_M}\Pi\text{div}_x\mathbf U~dxdt\right|\nonumber\\
&\leq C\varepsilon+C\int^\tau_0\|\Delta_x\Phi\|_{L^\infty(\Omega_M)}+\|\text{div}_xw_M\|_{L^\infty(\Omega_M)}~dt\nonumber\\
&\leq C\varepsilon+\frac{1}{M^2}.
\end{align}  }
For the term $R_{12}$, we have
\begin{align}
&\int^\tau_0\int_{\Omega_M}\rho(\mathbf U-\mathbf u)\partial_tw_M~dxdt\nonumber\\
=&\int^\tau_0\int_{\Omega_M}[\rho]_{ess}(\mathbf U-\mathbf u)\partial_tw_M~dxdt
+\int^\tau_0\int_{\Omega_M}[\rho]_{res}(\mathbf U-\mathbf u)\partial_tw_M~dxdt\nonumber\\
\leq &C\int^\tau_0\int_{\Omega_M}\rho|\mathbf U-\mathbf u|^2~dxdt
+\int^\tau_0\|\partial_tw_M\|_{L^2(\Omega_M)}~dt\nonumber\\
&\hspace{10pt}+\int^\tau_0\|\sqrt{[\rho]_{res}}\|_{L^{2\gamma}(\Omega_M)}\|\partial_tw_M\|_{L^p(\Omega_M)}\|\sqrt{\rho}(\mathbf U-\mathbf u)\|_{L^2(\Omega_M)}~dt\nonumber\\
\leq &C\int^\tau_0\int_{\Omega_M}\rho|\mathbf U-\mathbf u|^2~dxdt
+C(\varepsilon^m+\frac{1}{M^{m}}),\hspace{3pt}(m=1-\frac{1}{p}>2).
\end{align}
Then we turn to the term $R_3$:
\begin{align}
&\left|\frac{1}{\varepsilon^2}\int^\tau_0\int_{\Omega_M}\big{(}p(r)-p'(r)(r-\rho)-p(\rho)\big{)}\text{div}_x\mathbf U~dxdt\right|\nonumber\\
=&\left|\frac{1}{\varepsilon^2}\int^\tau_0\int_{\Omega_M}\big{(}p(r)-p'(r)(r-\rho)-p(\rho)\big{)}\text{div}_x(\nabla_x\Phi+\text{div}_xw_M)~dxdt\right|\nonumber\\
\leq&\int^\tau_0(\|\Delta_x\Phi\|_{L^\infty(\Omega_M)}+\|\text{div}_xw_M\|_{L^\infty(\Omega_M)})\mathcal{E}~dt.
\end{align}
{
It is easy to compute that
\begin{align*}
\nabla(H'(r)-H''(1)(r-1)-H'(1))&=\gamma\nabla_x\left(\frac{(1+\varepsilon s)^{\gamma-1}-1}{\gamma-1}-\varepsilon s\right)=\varepsilon\gamma\left((1+\varepsilon s)^{\gamma-2}-1\right)\nabla_xs
\end{align*}
and then we have
\begin{align}
|\nabla_x(H'(r)-H''(1)(r-1)-H'(1))|\leq C\varepsilon^2|s\nabla_xs|.
\end{align}
Therefore, we can insert the above equality into $R_2$:
\begin{align}
&\left|\frac{1}{\varepsilon^2}\int^\tau_0\int_{\Omega_M}\rho(\mathbf U-\mathbf u)\nabla_x\big{(}H'(r)-H''(1)(r-1)-H'(1)\big{)}~dxdt\right|\nonumber\\
\leq&\int^\tau_0\int_{\Omega_M}|\rho(\mathbf U-\mathbf u)||s\nabla_xs|~dxdt\nonumber\\
\leq&\frac 1 2\int^\tau_0\int_{\Omega_M}
\rho\left|\mathbf U-\mathbf u\right|^2 dxdt
+\frac 1 2\int^\tau_0\int_{\Omega_M}\rho|s|^2|\nabla_x s|^2
dxdt\nonumber\\
\leq &\frac 1 2\int^\tau_0\int_{\Omega_M}
\rho\left|\mathbf U-\mathbf u\right|^2 dxdt
+\frac 1 2\int^\tau_0\int_{\Omega_M}
\left(1+
\varepsilon\left[\frac{\rho-1}{\varepsilon}\right]_{ess}
+[\rho]_{res}\right)|s|^2|\nabla_x s|^2dxdt\nonumber\\
\leq &\frac 1 2\int^\tau_0\int_{\Omega_M}
\rho\left|\mathbf U-\mathbf u\right|^2 dxdt
+\frac{1}{2}\sup_{0\le t\le \tau}\|s(\cdot,t)\|^2_{L^2(\Omega_M)}\int_0^\tau\|\nabla_x s\|^2_{L^\infty(\Omega_M)}dt\nonumber\\
&+\frac 1 2 \varepsilon\sup_{0\le t\le\tau}
\left\|\left[\frac{\rho-1}{\varepsilon}\right]_{ess}\right\|_{L^2(\Omega_M)}\int_0^\tau\|s\|^2_{L^4(\Omega_M)}
\|\nabla_x s\|^2_{L^\infty(\Omega_M)}dt\nonumber\\
&+\frac 1 2 \sup_{0\le t\le\tau}
\left\|[\rho]_{res}\right\|_{L^\gamma(\Omega_M)}
\int_0^\tau\|s\|^2_{L^{\frac{2\gamma}{\gamma-1}}
(\Omega_M)}\|\nabla_x s\|^2_{L^\infty(\Omega_M)}dt\nonumber\\
\leq &\frac 1 2\int^\tau_0\int_{\Omega_M}
\rho\left|\mathbf U-\mathbf u\right|^2 dxdt
+\frac{C}{2}\sup_{0\le t\le \tau}\|s(\cdot,t)\|^2_{L^2(\Omega_M)}\int_0^\tau\left(
1+\frac t \varepsilon\right)^{-2}dt\nonumber\\
&+\frac 1 2 \varepsilon\sup_{0\le t\le\tau}
\left\|\left[\frac{\rho-1}{\varepsilon}\right]_{ess}\right\|_{L^2(\Omega_M)}\int_0^\tau
\left(1+\frac t \varepsilon\right)^{-1}
\left(1+\frac t \varepsilon\right)^{-2}
dt\nonumber\\
&+\frac 1 2 \sup_{0\le t\le\tau}
\left\|[\rho]_{res}\right\|_{L^\gamma(\Omega_M)}
\int_0^\tau\left(1+\frac t \varepsilon\right)^{-\frac{2}{\gamma}}
\left(1+\frac t \varepsilon\right)^{-2}dt\nonumber\\
\leq &\frac 1 2\int^\tau_0\int_{\Omega_M}
\rho\left|\mathbf U-\mathbf u\right|^2 dxdt
+C\varepsilon,
\end{align}
where we used the fact that for $\beta>1$, 
\[
\int_0^\tau \left(1+\frac{t}{\varepsilon}\right)^{-\beta}dt=\varepsilon\int_0^{\frac{\tau}{\varepsilon}} \left(1+u\right)^{-\beta}du\le\frac{\varepsilon}{\beta-1}.
\]
}
Next, we use \eqref{2.13} to estimate $R_4$ as 
{\begin{align}
&\left|\frac{1}{\varepsilon^2}\int^\tau_0\int_{\Omega_M}(r-\rho)H''(r)\text{div}_xw_M~dxdt\right|\nonumber\\
=&\left|\frac{1}{\varepsilon}\int^\tau_0\int_{\Omega_M}\frac{1-\rho}{\varepsilon}H''(r)\text{div}_xw_M~dxdt
+\frac{1}{\varepsilon}\int^\tau_0\int_{\Omega_M}sH''(r)\text{div}_xw_M~dxdt\right|\nonumber\\
\leq &\frac{1}{\varepsilon}\left|\int^\tau_0\left(-\left[\frac{\rho-1}{\varepsilon}\right]_{ess}
+\frac{1}{\varepsilon}\left[1\right]_{res}+\frac{1}{\varepsilon}\left[\rho\right]_{res}\right)\text{div}_x w_M~dt
\right|+\frac{1}{\varepsilon}\int^\tau_0\|\text{div}_xw_M\|_{L^{2}(\Omega_M)}\|s\|_{L^2(\Omega_M)}~dt\nonumber\\
\leq&\frac{1}{\varepsilon}\int_0^\tau\left\|\left[\frac{\rho-1}{\varepsilon}\right]_{ess}\right\|_{L^2(\Omega_M)}\|\text{div}_xw_M\|_{L^{2}(\Omega_M)}dt
+\frac{1}{\varepsilon^2}\int_0^\tau\left\|\left[1\right]_{res}\right\|_{L^1(\Omega_M)}\|\text{div}_xw_M\|_{L^{\infty}(\Omega_M)}dt\nonumber\\
&+\frac{1}{\varepsilon^2}\int_0^\tau\left\|\left[\rho^\gamma\right]_{res}\right\|_{L^1(\Omega_M)}
\|\text{div}_xw_M\|_{L^{\frac{\gamma}{\gamma-1}}(\Omega_M)}dt+\frac{1}{\varepsilon}\int^\tau_0\|\text{div}_xw_M\|_{L^{2}(\Omega_M)}\|s\|_{L^2(\Omega_M)}~dt\nonumber\\
\leq&C\left(\frac{1}{\varepsilon M}+\frac{1}{M^2}+\frac{1}{M^\frac{2}{\gamma}}\right).
\end{align}}
Similarly, we control $R_5$ as
{\begin{align}\label{4.14}
&\left|\frac{1}{\varepsilon}\int^\tau_0\int_{\Omega_M}(r-\rho)H''(r)\text{div}_x(s\mathbf U)~dxdt\right|\nonumber\\
=&\left|\int^\tau_0\int_{\Omega_M}\left(\frac{1-\rho}{\varepsilon}+s\right)H''(r)[\nabla_xs\cdot(\mathbf v+\nabla_x\Phi+w_M)+s\Delta_x\Phi+s\text{div}_xw_M]~dxdt\right|\nonumber\\
\leq & C\int^\tau_0\int_{\Omega_M}\left(|\frac{1-\rho}{\varepsilon}|+|s|\right)(|s|+|\nabla_xs|)
\big{(}|\mathbf v|+|\nabla_x\Phi|+|w_M|+|\Delta_x\Phi|+|\text{div}_xw_M|\big{)}~dxdt\nonumber\\
\leq&\eta(\epsilon, M),
\end{align} 
where $\eta(\epsilon, M)$ tends to zero by similar estimates as above;  for instance,
\begin{equation}
\int^\tau_0\int_{\Omega_M}|s||\nabla_x s||\text{div}_xw_M|dxdt
\leq \int^\tau_0\|s\|_{L^2(\Omega_M)}\|\nabla_x s\|_{L^2(\Omega_M)}
\|\text{div}_xw_M\|_{L^\infty(\Omega_M)}dt
\leq \frac{C}{M^2}.
\end{equation}
}

Now, we consider the drag term $R_7$. Recalling $r_1(\varepsilon)\rightarrow0$ as $\varepsilon\rightarrow0$ and \eqref{2.5}, the term $R_7$ is controlled as 
\begin{align}
\left|r_1(\varepsilon)\int^\tau_0\int_{\Omega_M}|\sqrt{\rho}\mathbf u|\sqrt{\rho}\mathbf u\cdot U~dxdt\right|
\leq \eta(\epsilon)\rightarrow0.
\end{align}
Finally, we turn to the most difficult term, the viscous one, using $\mathbf U=\mathbf v+\nabla_x\Phi+w_M$:

\begin{align}
\varepsilon^\alpha\int^\tau_0\int_{\Omega_M}\sqrt{\rho}\mathbb{S}_\mu:\nabla_x\mathbf U~dxdt
&=\varepsilon^\alpha\int^\tau_0\int_{\Omega_M}\left(\sqrt{\rho}\mathbb{S}_\mu:\nabla_x\mathbf v+\sqrt{\rho}\mathbb{S}_\mu:\nabla_x^2\Phi+
\sqrt{\rho}\mathbb{S}_\mu:\nabla_xw_M\right)~dxdt\nonumber\\
&=\sum^{3}_{i=1}R_{6i}.
\end{align}
Using the Cauchy-Schwarz inequality, we have
\begin{align*}
R_{61}\leq&\varepsilon^\alpha\int^\tau_0\left(\int_{\Omega_M}\rho|\nabla_x\mathbf v|^2~dx\right)^{\frac{1}{2}}
\left(\int_{\Omega_M}|\mathbb{S}_\mu|^2~dx\right)^{\frac{1}{2}}~dt\\
\leq&\varepsilon^\alpha\int^\tau_0\int_{\Omega_M}\rho|\nabla_x\mathbf v|^2~dxdt
+\frac{1}{8}\varepsilon^\alpha\int^\tau_0\int_{\Omega_M}|\mathbb{S}_\mu|^2~dxdt\\
=&R_{611}+R_{612},
\end{align*}
where $R_{612}$ is under control. We divide $R_{611}$ into two parts:
\begin{align}
R_{611}=&\varepsilon^\alpha\int^\tau_0\int_{\Omega_M}[\rho]_{ess}|\nabla_x\mathbf v|^2~dxdt+\varepsilon^\alpha\int^\tau_0\int_{\Omega_M}[\rho]_{res}|\nabla_x\mathbf v|^2~dxdt\nonumber\\
\leq&\varepsilon^{\alpha+1}\int^\tau_0\int_{\Omega_M}\left[\frac{\rho-1}{\varepsilon}\right]_{ess}|\nabla_x\mathbf v|^2~dxdt
+\varepsilon^{\alpha}\int^\tau_0\int_{\Omega_M}|\nabla_x\mathbf v|^2~dxdt\nonumber\\
&+\varepsilon^{\alpha}\int^\tau_0\|[\rho]_{res}\|_{L^\gamma(\Omega_M)}\left(\int_{\Omega_M}|\nabla_x\mathbf v|^{\frac{2\gamma}{\gamma-1}}~dx\right)^{\frac{\gamma}{\gamma-1}}~dt\nonumber\\
\leq&\varepsilon^{\alpha+1}\int^\tau_0\left\|\left[\frac{\rho-1}{\varepsilon}\right]_{ess}\right\|_{L^2(\Omega_M)}\|\nabla_x\mathbf v\|^2_{L^4(\Omega_M)}~dt
+\varepsilon^{\alpha}\int^\tau_0\int_{\Omega_M}|\nabla_x\mathbf v|^2~dxdt\nonumber\\
&+\varepsilon^{\alpha+\frac2\gamma}\int^\tau_0\left(\int_{\Omega_M}|\nabla_x\mathbf v|^{\frac{2\gamma}{\gamma-1}}~dx\right)^{\frac{\gamma}{\gamma-1}}~dt\nonumber\\
&\leq \varepsilon^{\alpha+1}+\varepsilon^\alpha+\varepsilon^{\alpha+\frac2\gamma}.
\end{align}
Similarly, one estimates $R_{62}$ as
\begin{align*}
R_{62}\leq&
\varepsilon^\alpha\int^\tau_0\left(\int_{\Omega_M}\rho|\nabla_x^2\Phi|^2~dx\right)^{\frac{1}{2}}\left(\int_{\Omega_M}|\mathbb{S}_\mu|^2~dx\right)^{\frac{1}{2}}~dt\nonumber\\
\leq&\varepsilon^\alpha\int^\tau_0\int_{\Omega_M}\rho|\nabla_x^2\Phi|^2~dxdt+\frac{1}{8}\varepsilon^\alpha\int^\tau_0\int_{\Omega_M}|\mathbb{S}_\mu|^2~dxdt\\
=&R_{621}+\frac{\varepsilon^\alpha}{8}\int^\tau_0\|\mathbb{S}_\mu\|^2_{L^2(\Omega_M)}~dt.
\end{align*}
A similar analysis gives rise to the following:
\begin{align}
R_{621}=&\varepsilon^\alpha\int^\tau_0\int_{\Omega_M}[\rho]_{ess}|\nabla_x^2\Phi|^2~dxdt+\varepsilon^\alpha\int^\tau_0\int_{\Omega_M}[\rho]_{res}|\nabla_x^2\Phi|^2~dxdt\nonumber\\
\leq&\varepsilon^{\alpha+1}\int^\tau_0\int_{\Omega_M}[\frac{\rho-1}{\varepsilon}]_{ess}|\nabla_x^2\Phi|^2~dxdt
+\varepsilon^{\alpha}\int^\tau_0\int_{\Omega_M}|\nabla_x^2\Phi|^2~dxdt\nonumber\\
&+\varepsilon^{\alpha}\int^\tau_0\|[\rho]_{res}\|_{L^\gamma(\Omega_M)}\|\nabla_x^2\Phi\|^2_{L^{\frac{2\gamma}{\gamma-1}}(\Omega_M)}~dt\nonumber\\
\leq&\varepsilon^{\alpha+1}\int^\tau_0\left\|\left[\frac{\rho-1}{\varepsilon}\right]_{ess}\right\|_{L^2(\Omega_M)}\|\nabla_x^2\Phi\|^2_{L^4(\Omega_M)}~dt
+\varepsilon^{\alpha}\int^\tau_0\|\nabla_x^2\Phi\|^2_{L^2(\Omega_M)}~dxdt\nonumber\\
&+\varepsilon^{\alpha+2}\int^\tau_0\|\nabla_x^2\Phi\|^2_{L^{\frac{2\gamma}{\gamma-1}}(\Omega_M)}~dt.
\end{align}
Using the Strichartz estimate \eqref{2.12} for $\nabla_x\Phi$, we have
\begin{align*}
\|\nabla_x^2\Phi\|_{L^{2}(\Omega_M)},\ \|\nabla_x^2\Phi\|_{L^{4}(\Omega_M)},\ \|\nabla_x^2\Phi\|_{L^{\frac{2\gamma}{\gamma-1}}(\Omega_M)}\leq C\left(1+\frac{t}{\varepsilon}\right)^{-1},
\end{align*}
then we deduce for some $m>0$
\begin{align}
R_{621}\leq\varepsilon^m\int^\tau_0\left(1+\frac{t}{\varepsilon}\right)^{-1}dt\leq\varepsilon^{m+1}\ln(1+\varepsilon^{-1}\tau)\rightarrow0, \hspace{3pt}\text{as}\hspace{3pt}\varepsilon\rightarrow0.
\end{align}

Once more, we apply the Cauchy-Schwarz inequality to get
\begin{align}
R_{63}\leq&\varepsilon^\alpha\int^\tau_0\left(\int_{\Omega_M}\rho|\nabla_xw_M|^2~dx\right)^{\frac{1}{2}}\left(\int_{\Omega_M}|\mathbb{S}_\mu|^2~dx\right)^{\frac{1}{2}}~dt\nonumber\\
\leq&\varepsilon^\alpha\int^\tau_0\int_{\Omega_M}\rho|\nabla_xw_M|^2~dxdt+\frac{1}{8}\varepsilon^\alpha\int^\tau_0\int_{\Omega_M}|\mathbb{S}_\mu|^2~dxdt\nonumber\\
=&R_{631}+R_{632},
\end{align}
then 
\begin{align}\label{4.21}
R_{631}=&\varepsilon^\alpha\int^\tau_0\int_{\Omega_M}[\rho]_{ess}|\nabla_xw_M|^2~dxdt+\varepsilon^\alpha\int^\tau_0\int_{\Omega_M}[\rho]_{res}|\nabla_xw_M|^2~dxdt\nonumber\\
\leq&\varepsilon^{\alpha+1}\int^\tau_0\int_{\Omega_M}\left[\frac{\rho-1}{\varepsilon}\right]_{ess}|\nabla_xw_M|^2~dxdt
+\varepsilon^{\alpha}\int^\tau_0\int_{\Omega_M}|\nabla_xw_M|^2~dxdt\nonumber\\
&+\varepsilon^{\alpha}\int^\tau_0\|[\rho]_{res}\|_{L^\gamma(\Omega_M)}\left(\int_{\Omega_M}|\nabla_xw_M|^{\frac{2\gamma}{\gamma-1}}~dx\right)^{\frac{\gamma}{\gamma-1}}~dt\nonumber\\
\leq&\varepsilon^{\alpha+1}\int^\tau_0\left\|\left[\frac{\rho-1}{\varepsilon}\right]_{ess}\right\|_{L^2(\Omega_M)}\|\nabla_xw_M\|^2_{L^4(\Omega_M)}~dt
+\varepsilon^{\alpha}\int^\tau_0\int_{\Omega_M}|\nabla_xw_M|^2~dxdt\nonumber\\
&+\varepsilon^{\alpha+2}\int^\tau_0\left(\int_{\Omega_M}|\nabla_xw_M|^{\frac{2\gamma}{\gamma-1}}~dx\right)^{\frac{\gamma}{\gamma-1}}~dt\nonumber\\
\leq& \eta_2(\epsilon,M),
\end{align}
where  $\eta_2(\varepsilon, M)\rightarrow0$ by virtue of similar estimates.

Putting \eqref{4.2}-\eqref{4.21} together, we conclude that
\begin{align*}
\mathcal{E}(\tau)-\mathcal{E}(0)+\varepsilon^\alpha\int^\tau_0\int_{\Omega_M}|\mathcal{S}_\mu|^2~dxdt\leq
\int^\tau_0(C+\eta_1(\varepsilon))\mathcal{E}~dt+\eta_2(\varepsilon, M),
\end{align*}
where $\eta_1(\varepsilon)$, $\eta_2(\varepsilon, M)\rightarrow0$, as $\varepsilon\rightarrow0$, $M\rightarrow+\infty$. Thus we have completed the proof of Theorem \ref{thm2.1}.
\qed
\vskip 0.5cm




\vskip 0.5cm
\noindent {\bf Acknowledgements}

\vskip 0.1cm
The authors thank Eduard Feireisl for helpful discussions and constructive suggestions. T.~Tang is partially supported by NSFC No.~12371246 and Qing Lan Project of Jiangsu Province. E.~Wiedemann is supported by the DFG Priority Programme SPP 2410 (project number 525716336). T.~Tang wishes to thank the University of Erlangen-N\"{u}rnberg for their hospitality during his stay in Germany, while E.~Wiedemann gratefully acknowledges the support of Yangzhou University during his visit to China.


\end{document}